\documentclass[10pt,reqno]{amsart}
\usepackage{amsfonts,amsmath,amssymb,amsbsy,amstext,amsthm}
\usepackage{accents,color,enumerate,enumitem,float,fullpage,parskip,verbatim}

\usepackage{url}
\usepackage[colorlinks=true,hyperindex, linkcolor=magenta, pagebackref=false, citecolor=cyan]{hyperref}
\usepackage[alphabetic,lite,backrefs]{amsrefs} 

\usepackage{eucal,bm,kpfonts,mathbbol}

\usepackage{tikz,tikz-cd}	
\usetikzlibrary{positioning, matrix, shapes}         								    				
\usetikzlibrary{arrows,calc,matrix}

\usepackage{lscape}

\usepackage{microtype}

\usepackage{titlesec}		
\titleformat{\section}[block]{\large\scshape\bfseries\filcenter}{\thesection.}{1em}{}		
\titleformat{\subsection}[hang]{\large\scshape\bfseries}{\thesubsection}{1em}{}			
\titleformat{\subsubsection}[hang]{\large\scshape\bfseries}{\thesubsubsection}{1em}{}			

\usepackage[titles]{tocloft}								     					
\usepackage{multirow}
\usepackage{array}
\usepackage{booktabs}
\newcolumntype{M}[1]{>{\centering\arraybackslash}m{#1}}
\newcolumntype{N}{@{}m{0pt}@{}}
\usepackage{diagbox}
\usepackage{cancel}

\newtheorem{lemma}{Lemma}[section]
\newtheorem{theorem}[lemma]{Theorem}
\newtheorem{prop}[lemma]{Proposition}

\newtheorem*{assumption}{Assumption}

\newtheorem{theoremalpha}{Theorem}

\theoremstyle{remark}

\newtheorem{remark}[lemma]{Remark}

\newcommand{\Span}{\operatorname{span}}

\newcommand{\conv}{\operatorname{conv}}

\newcommand{\Sym}{\operatorname{Sym}} 

\newcommand{\oO}{\operatorname{O}}

\newcommand{\inter}{\operatorname{int}}
\newcommand{\Nef}{\operatorname{Nef}}

\newcommand{\kk}{\mathbf k}

\renewcommand{\O}{\mathcal{O}}

\newcommand{\F}{\mathbb{F}}

\newcommand{\N}{\mathbb{N}}
\renewcommand{\P}{\mathbb{P}}

\newcommand{\R}{\mathbb{R}}

\newcommand{\Z}{\mathbb{Z}}

\title{The Quantitative Behavior of Asymptotic Syzygies for Hirzebruch Surfaces}

\author{Juliette Bruce}
\address{Department of Mathematics, University of Wisconsin, Madison, WI}
\email{\href{mailto:juliette.bruce@math.wisc.edu}{juliette.bruce@math.wisc.edu}}
\urladdr{\url{http://math.wisc.edu/~juliettebruce/}}

\thanks{The author was partially supported by the NSF GRFP under Grant No. DGE-1256259 and NSF grant DMS-1502553.}


\begin{document} 

\maketitle



\setcounter{section}{0}

In \cite{einErmanLazarsfeld15} Ein, Erman, and Lazarsfeld proposed a heuristic for the quantitative behavior of asymptotic syzygies: a Betti table that is ``sufficiently positive''  behaves approximately like the Betti table of a large Koszul complex. In particular, the entries of each row of such a Betti table should, after possibly rescaling, look like a normal Gaussian distribution. The goal of this note is to consider Ein, Erman, and Lazarsfeld's normality heuristic for a new class of examples, namely certain toric surfaces (including Hirzebruch surfaces) when the embedding line bundle grows in a semi-ample fashion. 

Our main results are twofold, and can be visualized in the case of $\P^{1}\times\P^{1}$ embedded by $\O_{\P^{1}\times\P^{1}}(d,2)$. Below we plot the entries of the two interesting rows of the Betti table for $d=3,5,10,$ and $20$.
\begin{figure}[H]
\includegraphics[scale=.85]{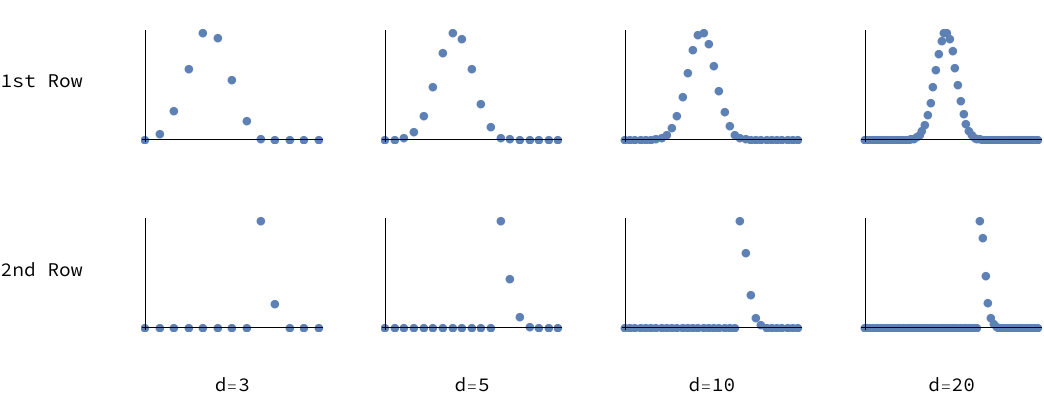}
\caption{Entries of the first row (top) and the second row (bottom) of the Betti table of $\P^1\times\P^1$.}
\label{fig:examples}
\end{figure}
In line with Ein, Erman, and Lazarsfeld's normality heuristic  Figure~\ref{fig:examples} shows that the first row of these Betti tables appears to be approaching a normal distribution. Our first result says this behavior generalizes to all Hirzebruch surfaces, which we denote by $\F_{t}$, with respect to the line bundle $\O_{\F_{t}}(d,2)=2E+dF$ where we identify $\Nef(\F_{t})\cong\N\langle E\rangle\oplus\N\langle F\rangle$ with $F$ being the fibre of the map to $\P^1$ and $E$ being the exceptional curve. Note the only previously known examples satisfying Ein, Erman, and Lazarsfeld's heuristic were smooth curves \cite{einErmanLazarsfeld15}*{Proposition~A} and random monomial ideals \cite{ermanYang18}*{Theorem~1.4}. 

Returning to Figure~\ref{fig:examples}, one sees that the second row of the Betti table of $\P^1\times\P^1$ embedding by $\O_{\P^{1}\times\P^{1}}(d,2)$ looks quite dissimilar to a normal distribution.  At best, the second row appears to converge to some fraction of a normal distribution. Our second main result shows that this failure of Ein, Erman, and Lazarsfeld's normality heuristic occurs for the second row of all Hirzebruch surfaces embedded by $\O_{\F_{t}}(d,2)$. 

As far as we are aware, this is the first known set of examples where something other than Ein, Erman, and Lazarsfeld's heuristic describes the quantitative behavior of asymptotic syzygies. We emphasize, however, that this is not a counterexample to \cite{einErmanLazarsfeld15}*{Conjecture~B} because that conjecture assumes the embedding line bundle grows in an ample fashion, and our examples are based on semi-ample growth.

To prove these theorems we build upon work of Lemmens who provided formulas for the graded Betti numbers for certain toric surfaces in terms of invariants of the associated polytopes. Our key observation is that by proving general results concerning the convergence of binomial distributions to normal distributions we can use these formulas to examine Ein, Erman, and Lazarsfeld's normality heuristic. 

Turning to the details, let $X$ be a projective variety of dimension $n$ over an arbitrary field $\kk$. Given a sequence of very ample line bundles $\{L_{d}\}_{d\in \N}$, we wish to study how the graded Betti numbers of $X$ behave asymptotically with respect to $L_{d}$ for $d\gg0$. That is we are interested in the syzygies of the section ring:
\[
R\left(X;L_d\right)\coloneqq\bigoplus_{k\in\Z}H^0\left(X, k\cdot L_d\right),
\]
as a module over $S=\Sym H^0(X, L_{d})\cong \kk[x_{0},x_{1},\ldots,x_{r_{d}}]$. Considering the graded minimal free resolution 
\[
\begin{tikzcd}[column sep = 3em]
0 & \lar R\left(X;L_{d}\right) & \lar F_{0} & \lar F_{1} & \lar \cdots & \cdots & \lar F_{r_{d}} & \lar 0\,,
\end{tikzcd}
\]
we let 
\[
K_{p,q}\left(X;L_{d}\right)\coloneqq\Span_{\kk}~\left\langle
\begin{matrix}
~\text{minimal generators of $F_p$} ~\\
\text{of degree $(p+q)$}
\end{matrix}
 \right\rangle
\]
be the finite dimensional $\kk$-vector space of minimal syzygies of homological degree $p$ and degree $(p+q)$. With this notation, $F_{p}$ is isomorphic to $\bigoplus_{q}K_{p,q}\left(X;L_{d}\right)\otimes_{\kk}S\left(-(p+q)\right)$. We write $k_{p,q}\left(X;L_{d}\right)$ for $\dim K_{p,q}\left(X;L_{d}\right)$, and then form the Betti table of $(X;L_{d})$ by placing $k_{p,q}\left(X;L_{d}\right)$ in the $(q,p)$-th spot as shown below:
\[
\begin{tabular}{c | ccc ccc}
          & 0                & 1                & $\cdots$  &  & $\cdots$ & $r_{d}$ \\ \hline
  $0$  & $k_{0,0}$   &  $k_{1,0}$  & $\cdots$  &  & $\cdots$ & $k_{r_{d},0}$ \\
  $1$  & $k_{0,1}$ & $k_{1,1}$  & $\cdots$ & & $\cdots$ & $k_{r_{d},1}$ \\
 \vdots & \vdots & \vdots &          & &            & \vdots \\
  $n$  & $k_{0,n}$ & $k_{1,n}$  & $\cdots$ & & $\cdots$ & $k_{r_{d},n}$ \\
\end{tabular}.
\]

In this set-up we may more precisely state Ein, Erman, and Lazarsfeld's heuristic as follows: if $\{L_{d}\}_{d\in\N}$ is a sequence of line bundles growing in positivity, then for any $q\in [1,n]$ there exists a function $F_{q}(d)$, depending on $X$, such that if $\{p_{d}\}_{d\in \N}$ is a sequence of non-negative integers such that
\begin{equation}\label{assumption}
\lim_{d\to\infty} \left[ p_{d}-\left(\frac{r_d}{2}+a\frac{\sqrt{r_d}}{2}\right)\right]=0,
\end{equation}
where $a\in \R$ is a fixed constant, then 
\[
F_{q}(d)\cdot k_{p_{d},q}\left(X;L_{d}\right) \to e^{-\frac{a^{2}}{2}}.
\]
Notice that the assumption that the sequence $\{p_{d}\}_{d\in\N}$ satisfies \eqref{assumption} is crucial. In particular, this change of coordinates is necessary for binomial distributions to converge to normal distributions.    As we will use it frequently, we take the time to label it here.

\begin{assumption}
We say that a sequence $\{p_{d}\}_{d\in\N}$ satisfies assumption $\bigstar$ if for some real number $a\in \R$:
\[
\lim_{d\to\infty} \left[ p_{d}-\left(\frac{r_d}{2}+a\frac{\sqrt{r_d}}{2}\right)\right]=0.
\]
\end{assumption}

We now state our main results concerning Ein, Erman, and Lazarsfeld's normality heuristic for Hirzebruch surfaces, which we denote by $\F_{t}$, embedded by the line bundle $\O_{\F_{t}}(d,2)$. 

\begin{theoremalpha}\label{thm:main1}
If $\{p_{d}\}_{d\in\N}$ is a sequence of non-negative integers satisfying $\bigstar$, then
\[
\frac{3\sqrt{2\pi}}{2^{r_{d}}\sqrt{r_{d}}}\cdot k_{p_{d},1}\left(\F_{t},\O_{\F_{t}}(d,2)\right)=e^{-\frac{a^{2}}{2}}\left(1+\oO\left(\frac{1}{\sqrt{r_{d}}}\right)\right).
\]
\end{theoremalpha}

\begin{theoremalpha}\label{thm:main2}
There does not exist a function $F_{2}(d)$ such that if $\{p_{d}\}_{d\in\N}$ is a sequence of non-negative integers satisfying $\bigstar$, then
\[
F_{2}(d)\cdot k_{p_{d},2}\left(\F_{t},\O_{\F_{t}}(d,2)\right)=e^{-\frac{a^{2}}{2}}\left(1+\oO\left(\frac{1}{\sqrt{r_{d}}}\right)\right).
\]
\end{theoremalpha}

Notice that  in this setting we are considering a slightly weaker positivity condition than was initially considered in \cite{einErmanLazarsfeld15}*{Conjecture~B}. In particular, the embedding line bundle is not growing in an ample fashion, but instead in a semi-ample fashion. This failure of Ein, Erman, and Lazarsfeld's heuristic when $q=2$ is likely related to the fact that the non-vanishing of asymptotic syzygies in the setting of semi-ample growth is quite nuanced. See for example \cite{bruce19}, where the author shows that the non-vanishing of asymptotic syzygies for products of projective spaces is not necessarily described by the non-vanishing theorems of Ein and Lazarsfeld \cite{einLazarsfeld12}. 


The note is structured as follows: In \S\ref{sec:toric-surfaces} we study the asymptotic distribution of graded Betti numbers for a family of toric surfaces. \S\ref{sec:toric-surfaces} also includes the proofs of Theorems~\ref{thm:main1} and \ref{thm:main2}. \S\ref{sec:technical-lemmas} contains technical results used in the proofs in  \S\ref{sec:toric-surfaces}.


\section*{Acknowledgements}  I would like to thank Daniel Erman, Kit Newton, Frank-Olaf Schreyer, and Benedek Valk\'{o}  for their helpful conversations and comments. \texttt{Macaulay2} \cite{M2} provided valuable assistance throughout this work. We also thank an anonymous referee for their many valuable comments and suggestions.




\section{Asymptotic Normality for Certain Toric Surfaces}\label{sec:toric-surfaces}

In this section we consider Ein, Erman, and Lazarsfeld's normality heuristic for certain toric surfaces, and prove slight generalizations of Theorems~\ref{thm:main1} and \ref{thm:main2}. Specifically we consider the toric surface $X_\delta$ whose associated normal fan $\Sigma_\delta\subset \R^2$ has four cones given by the rays $\{(1,0), (0,1), (0,-1), (-2,\delta)\}$, where $\delta\in \N$. 
\begin{center}
\begin{tikzpicture}[scale=.6]
   \fill[fill=gray!30]  (0,0) -- (2.8,0) -- (2.8,2.8) -- (0,2.8) -- cycle;
  \fill[fill=gray!50]  (0,0) -- (2.8,0) -- (2.8,-2.8) -- (0,-2.8) -- cycle;
  \fill[fill=gray!70]  (0,0) -- (0,-2.8) -- (-2.8,-2.8) -- (-2.8,2.8) -- (-1.91,2.8) -- cycle;
  \fill[fill=gray!90]  (0,0) -- (0,2.8) -- (-1.91,2.8) -- cycle;
  
  \draw[thin,dashed,gray!40](0,0)--(-3,0) node[anchor=north east]{};

  \draw[line width=1.35pt,black,-stealth](0,0)--(-2.0667,3) node[anchor=north east]{$\rho_1$};
  \draw[line width=1.35pt,black,-stealth](0,0)--(0,-3.1) node[anchor=north east]{$\rho_2$};
  \draw[line width=1.35pt,black,-stealth](0,0)--(3.1,0) node[anchor=south west]{$\rho_3$};
  \draw[line width=1.35pt,black,-stealth](0,0)--(0,3.1) node[anchor=north east]{$\rho_4$};
 
 \node at (-.6889,1)[circle, fill=black, scale=0.1] {$(-2,\delta)$};
 
\end{tikzpicture}
\end{center}

When $\delta$ is even, $X_{\delta}$ is isomorphic to the Hirzebruch surface $\F_{\delta/2}$. However, when $\delta$ is odd, $X_{\delta}$ is singular, with two $\Z/2\Z$-singularities \cite{coxLittleSchenck11}*{Proposition 10.1.2}.
 
For each ray $\rho_1,\rho_{2},\rho_{3},\rho_4$, there is a corresponding prime torus invariant divisor $D_{\rho_1},D_{\rho_{2}},D_{\rho_3},D_{\rho_4}$, which may be thought of as the irreducible components of $X_{\delta}\setminus T$, where $T \subset X_{\delta}$ is the torus. When $\delta$ is even, so that $X_{\delta}\cong \F_{\delta/2}$, these divisors are related to the generators of $\Nef(\F_{\delta/2})$ described in the introduction as follows: $D_{\rho_{1}}\sim D_{\rho_{3}}\sim F$, $D_{\rho_{2}}\sim E+(\delta/2)F$, and $D_{\rho_{4}}\sim E$. In particular, $\Nef(\F_{\delta/2})$ is generated by $D_{\rho_{1}}$ and $D_{\rho_{2}}$.

We are interested in the syzygies of $X_\delta$ with respect to the divisor $L_d=dD_{\rho_1}+2D_{\rho_2}$  when $\delta$ is even and $L_d=2dD_{\rho_1}+2D_{\rho_2}$ when $\delta$ is odd. The corresponding polytope for these divisors is
\[
\Delta_d=\conv\left\{(0,0), \; (d,0),\; (0,2), \; (d+\delta,2)\right\}.
\]
For example if $\delta=3$ and $d=2$ then $\Delta_d$ is the polytope below:
\begin{center}
\begin{tikzpicture}[scale=.40]

\node at (0,0) [circle, fill=black, scale=0.7] {};
\node at (2,0) [circle, fill=black, scale=0.7] {};
\node at (4,0) [circle, fill=black, scale=0.7] {};
\node at (6,0) [circle, fill=gray!30, scale=0.7] {};
\node at (8,0) [circle, fill=gray!30, scale=0.7] {};
\node at (10,0) [circle, fill=gray!30, scale=0.7] {};

\node at (0,2) [circle, fill=black, scale=0.7] {};
\node at (2,2) [circle, fill=black, scale=0.7] {};
\node at (4,2) [circle, fill=black, scale=0.7] {};
\node at (6,2) [circle, fill=black, scale=0.7] {};
\node at (8,2) [circle, fill=gray!30, scale=0.7] {};
\node at (10,2) [circle, fill=gray!30, scale=0.7] {};

\node at (0,4) [circle, fill=black, scale=0.7] {};
\node at (2,4) [circle, fill=black, scale=0.7] {};
\node at (4,4) [circle, fill=black, scale=0.7] {};
\node at (6,4) [circle, fill=black, scale=0.7] {};
\node at (8,4) [circle, fill=black, scale=0.7] {};
\node at (10,4) [circle, fill=black, scale=0.7] {};

\draw[line width=1.35pt,black,-](0,0) -- (0,4) -- (10,4);

\draw[line width=1.35pt,black,-](0,0) -- (4,0) -- (10,4);

\end{tikzpicture}
\end{center}

Throughout  the remainder of the paper, we write $n_{d}$ for $\#\left(\Delta_d\cap \Z^2\right)$ and $n^{(1)}_{d}$ for $\#\left(\inter(\Delta_{d})\cap \Z^2\right)$. For $d\geq1$ the divisor $L_{d}$ is very ample and defines an embedding $X_{\delta} \rightarrow \P H^{0}(L_{d})\cong \P^{r_{d}}$ where $r_d=n_{d}-1$. A straightforward  argument using Pick's theorem shows that
\[
r_{d}=3d+C_{\delta} \quad \text{and} \quad n_{d}^{(1)}=\frac{1}{3}r_{d}+E_{\delta}, \quad\quad \text{where} \quad\quad 
C_{\delta}\coloneqq \frac{3\delta}{2}+\frac{\gcd(\delta,2)}{2}+1 \quad \text{and} \quad E_{\delta}\coloneqq\frac{\gcd(\delta,2)}{3}-\frac{1}{3}.
\]

We now prove analogues of Theorems~\ref{thm:main1} and ~\ref{thm:main2} for $(X_{\delta};L_{d})$. As $X_{\delta}$ is isomorphic to $\F_{\delta/2}$ when $\delta$ is even, Theorems~\ref{thm:main1} and \ref{thm:main2} follow from these slightly more general theorems about $(X_{\delta};L_{d})$.

First, we show that Ein, Erman, and Lazarsfeld's normality heuristic accurately describes the quantitative behavior of the asymptotic linear syzygies of $(X_{\delta};L_{d})$.

\begin{theorem}\label{thm:toric-q1}
Set $F_{1}(d)=\frac{3\sqrt{2\pi}}{2^{r_{d}}\sqrt{r_{d}}}$. If $\{p_{d}\}_{d\in\N}$ is a sequence of non-negative integers satisfying $\bigstar$ then
\[
F_{1}(d)\cdot k_{p_{d},1}\left(X_{\delta};L_{d}\right)=e^{-\frac{a^{2}}{2}}\left(1+\oO\left(\frac{1}{\sqrt{r_{d}}}\right)\right).
\]
\end{theorem}

\begin{proof}
Using \cite{lemmens18}*{Corollary 5} together with the fact that $n_{d}=r_{d}+1$ and $n^{(1)}_{d}=\frac{1}{3}r_{d}+E_{\delta}$, we know that
\begin{align*}
k_{p_{d},1}\left(X_{\delta};L_{d}\right)
&=\max\left\{p_{d}-\frac{2}{3}r_{d}+E_{\delta}+1,\;\;0\right\}\binom{r_{d}-2}{p_{d}-1}+p_{d}\binom{r_{d}}{p_{d}+1}-\left(\frac{4}{3}r_{d}+E_{\delta}-1\right)\binom{r_{d}-2}{p_{d}-1}.
\end{align*}
By assumption $\bigstar$, we may replace $p_{d}$ with $\frac{r_{d}}{2}+a\frac{\sqrt{r_{d}}}{2}$ giving 
\begin{align}\label{eq2:toric-q1}
F_{1}(d)\cdot k_{p_{d},1}\left(X_{\delta};L_{d}\right)&
\sim F_{1}(d)\left(\max\left\{-\frac{r_{d}}{6}+a\frac{\sqrt{r_{d}}}{2}+E_{\delta}+1,\;\;0\right\}\binom{r_{d}-2}{p_{d}-1}+\left(\frac{r_{d}}{2}+a\frac{\sqrt{r_{d}}}{2}\right)\binom{r_{d}}{p_{d}+1}-\left(\frac{4}{3}r_{d}+E_{\delta}-1\right)\binom{r_{d}-2}{p_{d}-1}\right).
\end{align}
Proposition~\ref{lem:binom-to-norm} implies that for any constants $c_{1},c_{2}\in\Z$ and $c_{3}\in\R$ both $F_{1}(d)\cdot a\frac{\sqrt{r_{d}}}{2}\cdot \binom{r_{d}+c_{1}}{p_{d}+c_{2}}$ and $F_{1}(d)\cdot c_{3}\cdot  \binom{r_{d}+c_{1}}{p_{d}+c_{2}}$ tend to zero as $d\to \infty$. Hence we may ignore these terms in the above line, and rewrite \eqref{eq2:toric-q1} as
\begin{align*}\label{eq3:toric-q1}
F_{1}(d)\cdot k_{p_{d},1}\left(X_{\delta};L_{d}\right)&
\sim F_{1}(d)\cdot\left(\max\left\{-\frac{r_{d}}{6},\;\;0\right\}\binom{r_{d}-2}{p_{d}-1}+\frac{r_{d}}{2}\binom{r_{d}}{p_{d}+1}-\frac{4}{3}r_{d}\binom{r_{d}-2}{p_{d}}\right)
\sim \tfrac{3\sqrt{2\pi r_{d}}}{2^{r_{d}+1}}\binom{r_{d}}{p_{d}+1}-\tfrac{8\sqrt{2\pi r_{d}}}{2^{r_{d}+1}}\binom{r_{d}-2}{p_{d}-1}.
\end{align*}
The result now follows directly from Proposition~\ref{lem:binom-to-norm}. 
\end{proof}

Our second theorem in this section shows that the higher degree asymptotic syzygies of $(X_{\delta};L_{d})$ do not behave as suggested by Ein, Erman, and Lazarsfeld's normality heuristic. In particular, the entries in the $q=2$ row of the Betti table of $(X_{\delta};L_{d})$ do not converge to a normal distribution as $d\to \infty$. 

\begin{theorem}\label{thm:toric-q2}
There does not exist a function $F_{2}(d)$ such that if $\{p_{d}\}_{d\in\N}$ is a sequence of non-negative integers satisfying $\bigstar$ then
\[
F_{2}(d)\cdot k_{p_{d},2}\left(X_{\delta};L_{d}\right)=e^{-\frac{a^{2}}{2}}\left(1+\oO\left(\frac{1}{\sqrt{r_{d}}}\right)\right).
\]
\end{theorem}

\begin{proof}
Similar to the proof of Theorem~\ref{thm:toric-q1}, by using \cite{lemmens18}*{Corollary 5}, the fact that $n_{d}=r_{d}+1$ and $n^{(1)}_{d}=\frac{1}{3}r_{d}+E_{\delta}$, and assumption $\bigstar$ we see that
\begin{align*}
k_{p_{d},2}\left(X_{\delta};L_{d}\right)=\max\left\{\frac{-r_{d}}{6}+a\frac{\sqrt{r_{d}}}{2}+E_{\delta}+2,\;\;0\right\}\binom{r_{d}-2}{p_{d}}.
\end{align*}
However, $\frac{-r_{d}}{6}+a\frac{\sqrt{r_{d}}}{2}+E_{\delta}+3<0$ for $d\gg0$, and so for any function $F_{2}(d)$ we have that 
\[
F_{2}(d)\cdot k_{p_{d},2}\left(X_{\delta};L_{d}\right) = F_{2}(d)\cdot \max\left\{\frac{-r_{d}}{6}+a\frac{\sqrt{r_{d}}}{2}+E_{\delta}+3,\;\;0\right\}\binom{r_{d}-2}{p_{d}}=0.
\]
\end{proof}

\begin{remark}
Both Theorem~\ref{thm:toric-q1} and \ref{thm:toric-q2} only depend on the values of $k_{p,q}\left(X_{\delta};L_{d}\right)$ for $p$ in a neighborhood of $\frac{1}{2}r_{d}$, and hence are related in part to the study of Green's $N_{p}$-property \cite{schenck04,heringSchenckSmith06}. In particular, both theorems may also be deduced from Schenck's work showing that $k_{p,2}(X_{\delta};L_{d})=0$ for all $p\leq \#(\partial \Delta_d\cap \Z^2)-3$ \cite{schenck04}*{Corollary~2.1}. Since in the cases we are considering $\Delta_{d}$ is height two, $\#(\partial \Delta_d\cap \Z^2)$ is relatively large compared to $\frac{1}{2}r_{d}$. Using this one can deduce Theorem~\ref{thm:toric-q2}. One may then deduce Theorem~\ref{thm:toric-q1}  by noting that this vanishing implies that $k_{p,1}\left(X_{\delta};L_{d}\right)$ for $p$ around $\frac{1}{2}r_{d}$ can be deduced from the Hilbert function of $R(X_{\delta};L_{d})$. We thank the referee for pointing out this alternative argument. 
\end{remark}

\section{Technical Results}\label{sec:technical-lemmas}

Here we gather a series of technical results crucial to the proofs of our main theorems. The key result of this section is the following special case of the local central limit theorem. While standard in many probability texts, see \cite{durrett10}*{Theorem~3.5.3}, we take the time to prove it as this version of the local de Moivre-Laplace theorem, with precise constants and error terms, is crucial to proving our main theorems.

\begin{prop}\label{lem:binom-to-norm}
Suppose $\{r_{d}\}_{d\in\N}$ is a sequence such that $r_{d}\to\infty$ as $d\to\infty$. If there exists a sequence $\{p_{d}\}_{d\in\N}$ of non-negative integers satisfying $\bigstar$ then for any constants $c_{1},c_{2}\in \Z$:
\[
\frac{\sqrt{2\pi r_{d}}}{2^{r_{d}+1}}\binom{r_{d}+c_{1}}{p_{d}+{c_2}}=2^{c_{1}}e^{-\frac{a^{2}}{2}}\left[1+O\left(\frac{1}{\sqrt{r_{d}}}\right)\right].
\]
\end{prop}

Before proving this proposition we need the following lemma. 

\begin{lemma}\label{lem:tech-key-id}
Suppose $\{r_{d}\}_{d\in\N}$ is a sequence such that $r_{d}\to\infty$ as $d\to\infty$. For any $a\in \R$,
\begin{equation}\label{eq1:tech-key-id}
\left(\frac{r_{d}}{r_{d}+a\sqrt{r_{d}}}\right)^{\left(\frac{r_{d}}{2}+\frac{a\sqrt{r_{d}}}{2}\right)}\left(\frac{r_{d}}{r_{d}-a\sqrt{r_{d}}}\right)^{\left(\frac{r_{d}}{2}-\frac{a\sqrt{r_{d}}}{2}\right)}=e^{-\frac{a^2}{2}}\left[1+\oO\left(\frac{1}{\sqrt{r_{d}}}\right)\right].
\end{equation}
\end{lemma}

\begin{proof}
By taking $\log$ of the left hand side of equation~\eqref{eq1:tech-key-id} we see that
\begin{align}\label{eq2:tech-key-id}
-\log \left[\left(\frac{r_{d}}{r_{d}+a\sqrt{r_{d}}}\right)^{\left(\frac{r_{d}}{2}+\frac{a\sqrt{r_{d}}}{2}\right)}\left(\frac{r_{d}}{r_{d}-a\sqrt{r_{d}}}\right)^{\left(\frac{r_{d}}{2}-\frac{a\sqrt{r_{d}}}{2}\right)}\right]=
\frac{r_{d}+a\sqrt{r_{d}}}{2}\log\left(1+\frac{a}{\sqrt{r_{d}}}\right)+\frac{r_{d}-a\sqrt{r_{d}}}{2}\log\left(1-\frac{a}{\sqrt{r_{d}}}\right).
\end{align}
Using the Taylor expansion of $\log(1+x)$ the right hand side of equation~\eqref{eq2:tech-key-id} may be re-written as follows
\begin{align}\label{eq3:tech-key-id}
\frac{r_{d}+a\sqrt{r_{d}}}{2}\left(\frac{a}{\sqrt{r_{d}}}-\frac{a^2}{2r_{d}}+O\left(r_{d}^{-3/2}\right)\right)+\frac{r_{d}-a\sqrt{r_{d}}}{2}\left(-\frac{a}{\sqrt{r_{d}}}-\frac{a^2}{2r_{d}}-O\left(r_{d}^{-3/2}\right)\right).
\end{align}
Simplifying equation~\eqref{eq3:tech-key-id} and combining it with equation~\eqref{eq2:tech-key-id} shows that
\begin{align}
\log \left[\left(\frac{r_{d}}{r_{d}+a\sqrt{r_{d}}}\right)^{\left(\frac{r_{d}}{2}+\frac{a\sqrt{r_{d}}}{2}\right)}\left(\frac{r_{d}}{r_{d}-a\sqrt{r_{d}}}\right)^{\left(\frac{r_{d}}{2}-\frac{a\sqrt{r_{d}}}{2}\right)}\right]=
-\frac{a^2}{2}+O\left(\frac{1}{\sqrt{r_{d}}}\right).
\end{align}
The result now follows by exponentiating both sides of the above equation.
\end{proof}



\begin{proof}[Proof of Proposition~\ref{lem:binom-to-norm}]
Throughout the proof we write $\tilde{c}$ for $c_{1}-c_{2}$. A straightforward computation shows that $O(r_d)$, $O(r_d+c_{1})$, $O(p_{d}+c_{2})$ and $O(r_{d}-p_{d}+\tilde{c})$ are all equal, and so we will not distinguish between them. Using Stirling's formula for $n!$ we see that
\begin{align}\label{eq2:binom-to-norm}
\begin{split}
&\frac{\sqrt{2\pi r_{d}}}{2^{r_{d}+1}}\binom{r_{d}+c_{1}}{p_{d}+c_{2}}
\\ &=
2^{-(r_{d}+1)}\sqrt{\frac{r^2_d+c_{1}r_{d}}{(p_{d}+c_{2})(r_{d}-p_{d}+\tilde{c})}}
\left(\frac{r_{d}+c_{1}}{p_{d}+c_{2}}\right)^{c_{2}}\left(\frac{r_{d}+c_{1}}{r_{d}-p_{d}+\tilde{c}}\right)^{\tilde{c}}
\left(\frac{r_{d}+c_{1}}{p_{d}+c_{2}}\right)^{p_{d}}\left(\frac{r_{d}+c_{1}}{r_{d}-p_{d}+\tilde{c}}\right)^{r_{d}-p_{d}}\left(1+O\left(\frac{1}{r_d}\right)\right).
\end{split}
\end{align}
As $c_{1}$ and $c_{2}$ are constants and  $r_{d}$ and $p_{d}$ tend to infinity as $d\to\infty$, using assumption $\bigstar$ one can show that 
\[
\lim_{d\to\infty}\left(\frac{r_{d}+c_{1}}{p_{d}+c_{2}}\right)^{c_{2}}=2^{c_{2}}, \quad \lim_{d\to\infty}\left(\frac{r_{d}+c_{1}}{r_{d}-p_{d}+\tilde{c}}\right)^{\tilde{c}}=2^{\tilde{c}}, \quad \text{and} \quad \lim_{d\to\infty} \frac{r_{d}^2}{r_{d}p_{d}-p_{d}^{2}}=4.
\]
Using these limits, we see that \eqref{eq2:binom-to-norm} can be simplified to
\begin{align}\label{eq3:binom-to-norm}
\frac{\sqrt{2\pi r_{d}}}{2^{r_{d}+1}}\binom{r_{d}+c_{1}}{p_{d}+c_{2}}=2^{-r_{d}}2^{c_{1}}
\left(\frac{r_{d}+c_{1}}{p_{d}+c_{2}}\right)^{p_{d}}\left(\frac{r_{d}+c_{1}}{r_{d}-p_{d}+\tilde{c}}\right)^{r_{d}-p_{d}}\left(1+O\left(\frac{1}{r_{d}}\right)\right).
\end{align}
Now we show that we can reduce to the case when $c_{1}=c_{2}=\tilde{c}=0$. Towards this, notice that by the assumption $\bigstar$, we know that $r_{d}-p_{d}$ tends to infinity as $d\to\infty$. Combining this with the fact that both $r_{d}$ and $p_{d}$ also tend to infinity as $d\to\infty$, we see that
\begin{align*}
\lim_{d\to\infty}
\frac{
\left(\frac{r_{d}}{p_{d}}\right)^{p_{d}}\left(\frac{r_{d}}{r_{d}-p_{d}}\right)^{r_{d}-p_{d}}
}{
\left(\frac{r_{d}+c_{1}}{p_{d}+c_{2}}\right)^{p_{d}}\left(\frac{r_{d}+c_{1}}{r_{d}-p_{d}+\tilde{c}}\right)^{r_{d}-p_{d}}
}&=
\lim_{d\to\infty}\left[
\left(\frac{p_{d}+c_{2}}{p_{d}} \frac{r_{d}}{r_{d}+c_{1}}\right)^{p_{d}}\left(\frac{r_{d}}{r_{d}+c_{1}}\frac{r_{d}-p_{d}+\tilde{c}}{r_{d}-p_{d}}\right)^{r_{d}-p_{d}}\right]\\
&=\lim_{d\to\infty}\left(1+\frac{c_{2}}{p_{d}}\right)^{p_{d}} \left(1-\frac{c_{1}}{r_{d}+c_{1}}\right)^{r_{d}}
\left(1+\frac{\tilde{c}}{r_{d}-p_{d}}\right)^{r_{d}-p_{d}}=e^{c_{2}}\cdot e^{-c_{1}}\cdot e^{\tilde{c}}=1.
\end{align*}
In particular, we may re-write equation~\eqref{eq3:binom-to-norm} as
\begin{align}\label{eq5:binom-to-norm}
\frac{\sqrt{2\pi r_{d}}}{2^{r_{d}+1}}\binom{r_{d}+c_{1}}{p_{d}+c_{2}}=2^{-r_{d}}2^{c_{1}}
\left(\frac{r_{d}}{p_{d}}\right)^{p_{d}}\left(\frac{r_{d}}{r_{d}-p_{d}}\right)^{r_{d}-p_{d}}\left(1+O\left(\frac{1}{r_{d}}\right)\right).
\end{align}
Assumption $\bigstar$ allows us to substitute $\frac{r_{d}}{2}+\frac{a\sqrt{r_{d}}}{2}$ everywhere we see $p_{d}$ in \eqref{eq5:binom-to-norm}. Doing this gives
\begin{align}\label{eq6:binom-to-norm}
\frac{\sqrt{2\pi r_{d}}}{2^{r_{d}+1}}\binom{r_{d}+c_{1}}{p_{d}+c_{2}}=2^{c_{1}}
\left(\frac{r_{d}}{r_{d}+a\sqrt{r_{d}}}\right)^{\left(\frac{r_{d}}{2}+\frac{a\sqrt{r_{d}}}{2}\right)}\left(\frac{r_{d}}{r_{d}-a\sqrt{r_{d}}}\right)^{\left(\frac{r_{d}}{2}-\frac{a\sqrt{r_{d}}}{2}\right)}\left(1+O\left(\frac{1}{{d}}\right)\right),
\end{align}
from which the result follows using Lemma~\ref{lem:tech-key-id}.
\end{proof}


\end{document}